\documentclass[leqno,11pt]{amsart}
\usepackage{amsmath}
\usepackage{amssymb}
\usepackage{amsthm}

\usepackage{mathrsfs}
\usepackage{dsfont}
\usepackage{txfonts}
\usepackage{mathabx}

\usepackage{graphicx}
\usepackage[all]{xy}
\usepackage{pgf,tikz}
\usepackage{subfigure}
\usepackage{tikz-3dplot}
\usetikzlibrary{patterns,arrows}
\tikzset{
    partial ellipse/.style args={#1:#2:#3}{
        insert path={+ (#1:#3) arc (#1:#2:#3)}
    }
}

\usepackage{hyperref}
\usepackage{color}
\usepackage[shortlabels]{enumitem}

\newtheorem{theorem}{Theorem}[section]
\newtheorem{lemma}[theorem]{Lemma}
\newtheorem{corollary}[theorem]{Corollary}

\newtheorem{proposition}[theorem]{Proposition}

\theoremstyle{definition}

\theoremstyle{remark}

\newcommand{\tip}{\mathrm{tip}}

\newcommand{\bw}{\mathbf{w}}
\newcommand{\bg}{\mathbf{g}}
\newcommand{\bW}{\mathbf{W}}
\newcommand{\bG}{\mathbf{G}}

\newcommand{\eps}{\varepsilon}

\title[{A gradient estimate for the linearized translator equation}]{A gradient estimate for the linearized translator equation}

\author{Kyeongsu Choi, Robert Haslhofer, Or Hershkovits}

\begin{document}

\begin{abstract}
In this paper, we develop some analytic foundations for the linearized translator equation in $\mathbb{R}^4$, i.e. in the first dimension where the Bernstein property fails. This equation governs how the (noncompact) singularity models of the mean curvature flow in $\mathbb{R}^4$ fit together in a common moduli space. Here, we prove a gradient estimate, which gives a sharp bound for $W_v$, namely for the derivative of the variation field $W$ in the tip region. This serves as a substitute for the fundamental quadratic concavity estimate from Angenent-Daskalopoulos-Sesum, which has been crucial for controlling $Y_v$, namely the derivative of the profile function $Y$ in the tip region. Moreover, together with interior estimates by virtue of the linearized translator equation our gradient estimate implies a bound for $W_\tau$ as well. Hence, our gradient estimate also serves as substitute for Hamilton's Harnack inequality, which has played an important role for controlling $Y_\tau$ in the tip region.
\end{abstract}

\maketitle

\section{Introduction}

In this paper, we are concerned with the linearized translator equation
\begin{equation}\label{lin_trans_eq_intro}
\mathrm{div}(a_\phi  Du)+ b_\phi\cdot Du=f,
\end{equation}
where
\begin{equation}
a_\phi=\frac{\delta}{\sqrt{1+|D \phi|^2}}  -\frac{D\phi\otimes D\phi}{{(1+|D \phi|^2)^{3/2}}} , \qquad b_\phi=\frac{D \phi}{(1+|D \phi|^2)^{3/2}}.
\end{equation}
This equation arises as first variation of the translator equation
\begin{equation}\label{trans_eq_intro}
\mathrm{div}\left(\frac{D \phi}{\sqrt{1+|D \phi|^2}}\right)-\frac{1}{\sqrt{1+|D \phi|^2}}=0.
\end{equation}
We recall that translators model slowly forming singularities under mean curvature flow, see e.g. \cite{Hamilton_Harnack,HuiskenSinestrari_convexity,White_nature,HaslhoferKleiner_meanconvex}. In particular, it is known that all translators that arise as blowup limits of mean-convex mean curvature flow are given by convex entire solutions of \eqref{trans_eq_intro}. We also recall that the analytic property of being a convex entire graph is equivalent to the geometric property that $M=\mathrm{graph}(\phi)\subset \mathbb{R}^N$ is noncollapsed \cite{BN_noncollapsing,BLL}. In pioneering work \cite{Wang_convex}, Wang on the one hand proved that every noncollapsed translator in $\mathbb{R}^3$ is the unique rotationally symmetric bowl from \cite{AltschulerWu}, and on the other hand for every $N\geq 4$    constructed nontrivial examples of noncollapsed translators in $\mathbb{R}^N$, i.e. examples that are neither rotationally symmetric nor split off a line. Later, a more detailed construction of Wang's examples has been given by Hoffman-Ilmanen-Martin-White \cite{HIMW}. In particular, in $\mathbb{R}^4$ one obtains a one-parameter family of examples $\{M_{\kappa}\}_{\kappa\in (0,1/3)}$, called the oval-bowls, which are illustrated in Figure \ref{figure_oval_bowls}.\\

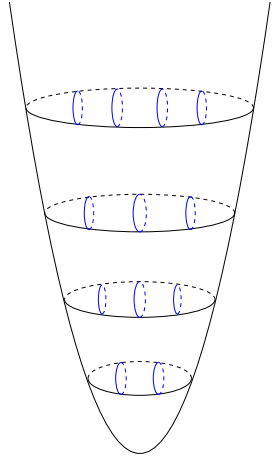
\begin{figure}
\scalebox{0.5}{
\begin{tikzpicture}[x=1cm,y=1cm] \clip(-4,12) rectangle (4,-1);
\draw [samples=100,rotate around={0:(0,0)},xshift=0cm,yshift=0cm,domain=-8:8)] plot (\x,{(\x)^2});
\draw [dashed] (0,2) [partial ellipse=0:180:2.75*0.5cm and 0.9*0.5cm];
\draw (0,2) [partial ellipse=180:360:2.75*0.5cm and 0.9*0.5cm];
\draw [dashed] (0,4.1) [partial ellipse=0:180:4*0.5cm and 0.95*0.5cm];
\draw (0,4.1) [partial ellipse=180:360:4*0.5cm and 0.95*0.5cm];
\draw [dashed] (0,6.4) [partial ellipse=0:180:5.05*0.5cm and 1*0.5cm];
\draw (0,6.4) [partial ellipse=180:360:5.05*0.5cm and 1*0.5cm];
\draw [dashed] (0,9.2) [partial ellipse=0:180:6.05*0.5cm and 1.05*0.5cm];
\draw (0,9.2) [partial ellipse=180:360:6.05*0.5cm and 1.05*0.5cm];
\draw [color=blue] (0,4.1) [partial ellipse=90:270:0.3*0.5cm and 0.95*0.5cm];
\draw [color=blue,dashed] (0,4.1) [partial ellipse=-90:90:0.3*0.5cm and 0.95*0.5cm];
\draw [color=blue] (-2*0.5,4.1) [partial ellipse=90:270:0.2*0.5cm and 0.77*0.5cm];
\draw [color=blue,dashed] (-2*0.5,4.1) [partial ellipse=-90:90:0.2*0.5cm and 0.77*0.5cm];
\draw [color=blue] (2*0.5,4.1) [partial ellipse=90:270:0.2*0.5cm and 0.77*0.5cm];
\draw [color=blue,dashed] (2*0.5,4.1) [partial ellipse=-90:90:0.2*0.5cm and 0.77*0.5cm];
\draw [color=blue] (0,6.4) [partial ellipse=90:270:0.35*0.5cm and 1*0.5cm];
\draw [color=blue,dashed] (0,6.4) [partial ellipse=-90:90:0.35*0.5cm and 1*0.5cm];
\draw [color=blue] (-2.7*0.5,6.4) [partial ellipse=90:270:0.25*0.5cm and 0.88*0.5cm];
\draw [color=blue,dashed] (-2.7*0.5,6.4) [partial ellipse=-90:90:0.25*0.5cm and 0.88*0.5cm];
\draw [color=blue] (2.7*0.5,6.4) [partial ellipse=90:270:0.25*0.5cm and 0.88*0.5cm];
\draw [color=blue,dashed] (2.7*0.5,6.4) [partial ellipse=-90:90:0.25*0.5cm and 0.88*0.5cm];
\draw [color=blue] (-1*0.5,2) [partial ellipse=90:270:0.28*0.5cm and 0.85*0.5cm];
\draw [color=blue,dashed] (-1*0.5,2) [partial ellipse=-90:90:0.28*0.5cm and 0.85*0.5cm];
\draw [color=blue] (1*0.5,2) [partial ellipse=90:270:0.28*0.5cm and 0.85*0.5cm];
\draw [color=blue,dashed] (1*0.5,2) [partial ellipse=-90:90:0.28*0.5cm and 0.85*0.5cm];
\draw [color=blue] (-1.2*0.5,9.2) [partial ellipse=90:270:0.28*0.5cm and 1*0.5cm];
\draw [color=blue,dashed] (-1.2*0.5,9.2) [partial ellipse=-90:90:0.28*0.5cm and 1*0.5cm];
\draw [color=blue] (1.2*0.5,9.2) [partial ellipse=90:270:0.28*0.5cm and 1*0.5cm];
\draw [color=blue,dashed] (1.2*0.5,9.2) [partial ellipse=-90:90:0.28*0.5cm and 1*0.5cm];
\draw [color=blue] (-3.3*0.5,9.2) [partial ellipse=90:270:0.25*0.5cm and 0.9*0.5cm];
\draw [color=blue,dashed] (-3.3*0.5,9.2) [partial ellipse=-90:90:0.25*0.5cm and 0.9*0.5cm];
\draw [color=blue] (3.3*0.5,9.2) [partial ellipse=90:270:0.25*0.5cm and 0.9*0.5cm];
\draw [color=blue,dashed] (3.3*0.5,9.2) [partial ellipse=-90:90:0.25*0.5cm and 0.9*0.5cm];
\end{tikzpicture}}
\caption{The oval-bowls $\{M_{\kappa}\}_{\kappa\in(0,1/3)}$ are noncollapsed translators in $\mathbb{R}^4$, whose level sets look like 2d-ovals in $\mathbb{R}^3$. The principal curvatures at the origin are $(\kappa,\tfrac{1-\kappa}{2},\tfrac{1-\kappa}{2})$.}\label{figure_oval_bowls}
\end{figure}

The linearized translator equation played an important role in our recent classification of noncollapsed translators in $\mathbb{R}^4$, see \cite{CHH_translators,CHH_lin_trans}.
Loosely speaking, while the analysis of individual translators is of course based on the bona-fide translator equation \eqref{trans_eq_intro}, to understand how these different translators fit together in a common moduli space, and in particular to show that the oval-bowls are uniquely determined by their tip-curvature $\kappa$, we had to develop a Fredholm theory for the linearized translator equation \eqref{lin_trans_eq_intro}.\\

The linearized equation is less geometric. Hence, several key estimates that are well established in the context of \eqref{trans_eq_intro}, specifically the avoidance principle \cite{Ilmanen_book}, the shrinker foliation from \cite{ADS1}, the quadratic concavity estimate from \cite{ADS2} and Hamilton's Harnack inequality \cite{Hamilton_Harnack}, are not at all obvious in the context of \eqref{lin_trans_eq_intro}. In \cite{CHH_lin_trans}, we already established two fundamental estimates, namely an upper-lower estimate and an inner-outer estimate, which serve as substitute for the former two. The purpose of the present paper is to establish an estimate that serves as substitute for the latter two.\\

To describe our results, let us first recall the setting from \cite{CHH_translators,CHH_lin_trans}. Let $M=\mathrm{graph}(\phi)\subset \mathbb{R}^4$ be an oval-bowl. We normalize, such that $\phi\geq 0$ with equality at the origin, and such that the $\mathrm{SO}_2$-symmetry is in the $x_2x_3$-variables. The description in terms of $\phi$ is most useful in the cap region $\{ \phi \leq C\}$, where $|D\phi|$ is bounded. To analyze the level sets $M\cap \{x_4=-t\}$ for $t\ll 0$ it is better to work with the cylindrical profile function $V$ and the tip profile function
$X$, which are implicitly defined by
\begin{equation}
\phi(x,V(x,t),0)=-t,\qquad \phi(X(\upsilon,t),\upsilon,0)=-t.
\end{equation}
Similarly, for the linearized translator equation the description in terms of $u$, which captures the variation in $x_4$-direction, is most useful in the cap region. To analyze the linearized translator equation for $x_4\gg 1$ it is better to work with the cylindrical variation $\bw$ and the tip variation $\bW$, which are defined by
\begin{equation}
\bw(x,t)=-V_t(x,t)\, u(x,V(x,t),0),\qquad \bW(\upsilon,t)=-X_t(\upsilon,t)\, u(X(\upsilon,t),\upsilon,0),
\end{equation}
and which capture the variation in $x_2x_3$-direction and in $x_1$-direction, respectively. In a similar vain, the inhomogeneity $f$ transforms to cylindrical and tip gauge according to the formulas
\begin{equation}\label{bg_def_intro}
\bg(x,t) =\sqrt{1+V_x^2+V_t^2}\, f(x,V(x,t),0),\qquad
\bG(\upsilon,t)=\sqrt{1+X_{\upsilon}^2+X_t^2}\, f(X(\upsilon,t),\upsilon,0).
\end{equation}
Most of the time we will actually work with the renormalized versions of these functions, specifically with
\begin{equation}
w(y,\tau)= e^{\frac{\tau}{2}}\bw\big(e^{-\frac{\tau}{2}}y,-e^{-\tau}\big),\qquad g(y,\tau)= e^{-\frac{\tau}{2}}\bg\big(e^{-\frac{\tau}{2}}y,-e^{-\tau}\big).
\end{equation}
In terms of $w$ and $g$ the linearized translator equation takes the form
\begin{equation}\label{eq_intro_w}
-w_\tau+\left(1-\frac{v_{y}^2}{1+ v_y^2}\right)w_{yy}-\left(\frac{y}{2}+\frac{2v_{y}v_{yy}}{(1+v_{y}^2)^2}\right) w_y+\left(1+\frac{2-v^2}{2v^2}\right)w+e^\tau \mathcal{F}w=g,
\end{equation}
where $v(y,\tau)=e^{\tau/2}V(e^{-\tau/2}y,-e^{-\tau})$ denotes the renormalized profile function. Finally, setting
\begin{equation}
W(v,\tau)= e^{\frac{\tau}{2}}\bW\big(e^{-\frac{\tau}{2}}v,-e^{-\tau}\big),\qquad G(v,\tau)= e^{-\frac{\tau}{2}}\bG\big(e^{-\frac{\tau}{2}}v,-e^{-\tau}\big),
\end{equation}
in tip gauge the linearized translator equation takes the form
\begin{equation}\label{evolve_inhom_tip_intro}
-W_\tau+\frac{W_{vv}}{1+Y_{v}^2}+\left(\frac{1}{v}-\frac{v}{2}-2\frac{Y_{vv}Y_{v}}{(1+Y_{v}^2)^2}\right)W_v+\frac{1}{2}W+e^{\tau} {\mathcal{F}}W=G\, ,
\end{equation}
where $Y(\cdot,\tau)$ denotes the renormalized version of $X$, in other words the inverse function of $v(\cdot,\tau)$.\footnote{The precise formulas for the terms $\mathcal{F}w$ and $\mathcal{F}W$ can be found in \cite[Section 3]{CHH_lin_trans}, though for our purpose it suffices to know both operators $\mathcal{F}$ are second order linear differential operators whose coefficients grow at most polynomially in $\tau$.}\\

Now, for any $h<\infty$, the hypersurface $M\cap \{x_4<h\}$ can be expressed as graph over a domain $\Omega_h\subset\mathbb{R}^3$. Denote by $C^{k-2,\alpha}(\Omega_h/S^1)$ the space of all $f\in C^{k-2,\alpha}(\Omega_h)$ that are $\mathrm{SO}_2$-symmetric in the $x_2x_3$-variables. To develop the Fredholm theory in \cite{CHH_lin_trans}, given any $h<\infty$ and $f\in C^{k-2,\alpha}(\Omega_h/S^1)$, we considered the Dirichlet problem
\begin{equation}\label{bdval_prob}
    \begin{cases}
      \mathrm{div}(a_\phi  Du)+ b_\phi\cdot Du  =f & \text{in  $\Omega_h$}\\
      u =0 &  \text{on $\partial \Omega_h$},\\
    \end{cases}       
\end{equation}
and established estimates that are uniform in $h$. Specifically, we proved an upper-lower estimate and and inner-outer estimate, which allow to propagate $L^\infty$-information. In particular, the inner-outer estimate allows to propagate smallness in the parabolic region $|y|\leq \ell$ to smallness in the intermediate and tip region. Loosely speaking, the estimate says if $|w|\leq A/|\tau|^{1+\mu}$ in the parabolic region, then $|w|\leq CA(\sqrt{2}-v)/|\tau|^\mu$ in the intermediate region and $|W|\leq CA|\tau|^{1/2-\mu}$ in the tip region $v\leq \theta$. More precisely, fixing a small constant $\theta >0$, a large constant $\ell<\infty$, and a large height $h_0<\infty$, such that the asymptotics from \cite{CHH_translators} hold at all $\tau\leq\tau_0=\log(h_0)$, the statement is:

\begin{theorem}[inner-outer estimate, \cite{CHH_lin_trans}]\label{inner_outer_intro}
Let $\mu\geq 0$. Suppose that $A<\infty$ is a constant such that
\begin{equation}\label{lower_bd_basic_sup_intro}
\sup_{\tau\in [-\log(h)+1,\tau_0]}  |\tau|^{1+\mu}\sup_{|y|\leq \ell}|w(y,\tau)|+ \sup_{\tau\in [-\log (h),-\log(h)+1]}  |\tau|^{1/2+\mu}\sup_{|y|\leq \ell}|w(y,\tau)| +
 \sup_{x\in \partial \Omega_{h_0}}|u(x)|\leq A,
\end{equation}
and suppose that for all $\tau\in [-\log (h),\tau_0]$ we have
\begin{equation}\label{g_growth_basic_sup_intro}
|\tau|^{\mu}\!\!\!\sup_{y \in \left[\ell,Y(\theta,\tau)\right]}    (\sqrt{2}-v)^{-2}|g(y,\tau)| +
|\tau|^{\mu}\sup_{v\leq \theta} \left(|\tau|^{1/2}+\frac{1}{v^3}\min\left(1,v^2|\tau|/\ell^2\right)\right)^{-1} |G(v,\tau)| \leq A.
\end{equation}
Then, there is some $C=C(\mu,\tau_0,\ell,\theta)<\infty$, such that for all $\tau\in [-\log (h)^{1/2},\tau_0]$ we get
\begin{equation}
|\tau|^{\mu}\!\!\!\sup_{y \in \left[\ell,Y(\theta,\tau)\right]}  (\sqrt{2}-v)^{-1}|w(y,\tau)| +
 |\tau|^{\mu-1/2} \sup_{v\leq \theta} |W(v,\tau)| \leq CA.
\end{equation}
\end{theorem}

In particular, Theorem \ref{inner_outer_intro} (inner-outer estimate) serves as a substitute for the shrinker foliation from \cite{ADS1}, which has been crucial for propagating information from the parabolic region to the intermediate and tip region, see e.g. \cite{ADS2,BC,BC2,CHH,CHHW,CHH_translators,DH_shape,DuZhu,CDDHS}.\\

Here, we establish a gradient estimate, which gives sharp $C^1$-control in the tip region. Loosely speaking, our estimate says that $|W_v|\leq CA(1+v|\tau|^{1/2})/|\tau|^{\mu}$ in the tip region. The precise statement is as follows:

\begin{theorem}[gradient estimate]\label{gradient_intro}
Let $\mu\geq 0$. In addition to \eqref{lower_bd_basic_sup_intro} and \eqref{g_growth_basic_sup_intro}, suppose that for all $\tau\in [-\log (h)^{1/2}+1,\tau_0]$ we have
\begin{equation}\label{grad_a2_intro}
|\tau|^{1/2+\mu}\!\!\! \sup_{y \in \left[\ell,Y(\theta,\tau)\right]}\,\, (\sqrt{2}- v)^{-3/2} |g_y(y,\tau)|+|\tau|^{\mu}\sup_{v\leq \theta}\,\, v^2(1+v^2|\tau|^{1/2})^{-1}  |G_v(v,\tau)| \leq A.
\end{equation}
 Then, there is some $C=C(\mu,\tau_0,\ell,\theta)<\infty$, such that for all $\tau\in [-\log(h)^{1/4},\tau_0]$ we get
\begin{equation}
|\tau|^{1/2+\mu}\!\!\! \sup_{y \in \left[\ell,Y(\theta,\tau)\right]}\,\, (\sqrt{2}- v)^{-1/2} |w_y(y,\tau)|
+|\tau|^{\mu}\sup_{v\leq \theta}\,\,(1+v|\tau|^{1/2})^{-1}|W_v(v,\tau)| \leq CA.
\end{equation}
\end{theorem}

Theorem \ref{gradient_intro} (gradient estimate) allows us to control $W_v$ in the tip region and thus serves as a substitute for the quadratic concavity estimate from \cite{ADS2}, which has been crucial for controlling $Y_v$ in the tip region, see e.g. \cite{ADS2,DH_ovals,CHH_translators,CDDHS,CDZ}. Let us also point out that together with interior estimates by virtue of the linearized translator equation the gradient estimate implies a bound for $W_\tau$ as well. Hence, the gradient estimate also serves as substitute Hamilton's Harnack inequality \cite{Hamilton_Harnack}, which is important for controlling $Y_\tau$ in the tip region, see again \cite{ADS2,DH_ovals,CHH_translators,CDDHS,CDZ}.\\

Let us give a brief outline of our approach. Since $W$ satisfies \eqref{evolve_inhom_tip_intro}, its derivative $W_v$ satisfies
\begin{equation}\label{LpWv}
L_{\tip}'[W_v]+e^\tau (\mathcal{F}W)_v=G_v,
\end{equation}
where the linear operator $L_{\tip}'$ is given by the formula
\begin{equation}\label{Ltip_Der}
L_{\tip}'[P]=-P_{\tau}+\frac{P_{vv}}{1+Y_{v}^2}+\left(\frac{1}{v}-\frac{v}{2}-4\frac{Y_{vv}Y_{v}}{(1+Y_{v}^2)^2}\right)P_v-\left(\frac{1}{v^2}+2\left(\frac{Y_{vv}Y_{v}}{(1+Y_{v}^2)^2}\right)_v\right)P.
\end{equation}
To prove our gradient estimate we will construct a suitable supersolution for the operator $L_{\tip}'$. We do this by patching together carefully designed supersolutions in the different regions, specifically the soliton region $v\leq \ell |\tau|^{-1/2}$, the collar region $\ell |\tau|^{-1/2}\leq v \leq \theta$, and the cylindrical region $v\geq \theta$. Since the coefficients of $L_{\tip}'$ depend on $Y$, this requires very precise estimates for the profile function $Y$.\\

In the soliton region, the zoomed in profile function is modelled by the bowl soliton, so we first derive some precise estimates for the bowl soliton. Recall that the bowl soliton is given as graph of rotation of a function $\varphi(r)$, which satisfies $\varphi(r)\sim \tfrac{1}{4}r^2$ for $r\to 0$ and $\varphi(r)\sim \tfrac{1}{2}r^2$ for $r\to \infty$. To capture how the coefficient of the quadratic term transitions from $\tfrac14$ to $\tfrac12$, we consider the macroscopic curvature function $K(r)=\varphi'(r)/r$. It is easy to see that $\tfrac12 \leq K \leq 1$, but to obtain good enough control of the coefficients of $L_{\tip}'$ we have to derive way more precise estimates that capture how exactly $K$ transitions from $\tfrac 12$ to $1$.\\

Finally, to obtain good enough control for the coefficients of $L_{\tip}'$ in the cylindrical region, we derive an enhanced quadratic concavity estimate for the cylindrical profile function $v$. Recall that the fundamental quadratic concavity estimate from Angenent-Daskalopoulos-Sesum \cite{ADS2} says that $(v^2)_{yy}\leq 0$, and a similar quadratic concavity estimate for translators up to exponential errors has been established in \cite{CHH_translators}. Here, building on the intuition that $v$ is well approximated by an ellipse and
pushing further the methods from our paper \cite{CHH_profile}, we improve this to $(v^2)_{yy}\leq -(2-\eps)/|\tau|$ in a suitable region.\\

This paper is organized as follows. In Section \ref{sec_bowl}, we derive the precise profile estimates for the bowl. In Section \ref{sec_enh}, we derive the enhanced quadratic concavity estimate for the cylindrical profile function. Finally, in Section \ref{sec_grad}, we prove the gradient estimate for the linearized translator equation.\\

Throughout this paper, we fix a small constant $\theta>0$, a large constant $\ell<\infty$, and a very negative constant $\tau_0>-\infty$. By convention, these constants can be adjusted at finitely many instances.

\bigskip

\noindent\textbf{Acknowledgments.} We thank the referee for helpful comments.
KC has been supported by the KIAS Individual Grant MG078902, an Asian Young Scientist Fellowship, and the National Research Foundation (NRF) grants  RS-2023-00219980 and RS-2024-00345403 funded by the Korea government (MSIT). RH has been supported by the NSERC Discovery Grant RGPIN-2023-04419. OH has been supported by ISF grant 437/20. This project has received funding from the European Research Council (ERC) under the European Union's Horizon 2020 research and innovation program, grant agreement No 101116390.

\bigskip

\section{Precise profile estimates for the bowl soliton}\label{sec_bowl}

Let $\psi=\varphi'$ be the derivative of the bowl profile function with tip speed $1$, namely the solution of the initial value problem 
\begin{equation}\label{bowl_eq_1}
\frac{\psi'}{1+\psi^2}+\frac{\psi}{r}=1, \qquad \psi(0)=0.
\end{equation}

We are interested in the behaviour of the macroscopic curvature function
\begin{equation}
K(r)=\frac{\psi(r)}{r}.
\end{equation}

To begin with, note that since $r/2$ and $r$ are sub/super solutions to the equation \eqref{bowl_eq_1}, we have
\begin{equation}
\frac{1}{2} \leq K \leq 1.
\end{equation}

Next, we observe that the function $K$ is monotone, namely
\begin{equation}
rK' =  \psi'-K\geq 0.
\end{equation}
Indeed, if we had $\psi'(r_0)<K(r_0)$ at some $r_0>0$ then the linear function $L(r)=K(r_0)r$ would be a strict supersolution to \eqref{bowl_eq_1} on $[0,r_0]$, which is impossible since $\psi(0)=L(0)$ and $\psi(r_0)=L(r_0)$.

Moreover, it is easy to see that $K$ has the asymptotic expansions
\begin{equation}\label{gconvzero}
K(r)=\frac{1}{2}+\frac{r^2}{32}+O(r^4)\,\,\, \textrm{for } r\rightarrow 0\qquad \textrm{and}\qquad K(r)=1-\frac{1}{r^2}+O\left(\frac{1}{r^4}\right) \,\,\, \textrm{for } r\rightarrow \infty,
\end{equation}
see e.g. \cite[Section 2]{CSS}, though these expansions will actually not be used in our proof. For our purpose, we need more quantitative bounds. To state these bounds let us abbreviate
\begin{equation}
F(r,a|\rho,\alpha)=1-\frac{1-\left[1-(1-\alpha)a^2\rho^2\right]e^{-\frac{a^2}{2}(r^2-\rho^2)}}{a^2r^2}.
\end{equation}

\begin{proposition}[quantitative bounds]\label{app_prop}
Given any $h>0$, setting $r_i=i\cdot h$ we have the bounds
\begin{equation}\label{table_eq}
K(r)\geq a_i \,\,\, \textrm{for } r\geq r_i  \quad\textrm{and}\quad K(r)\leq b_i \,\,\, \textrm{for } r\leq r_i,
\end{equation}  
where the numbers $a_i$ and $b_i$ are recursively defined by
\begin{equation}
a_{i+1}=F(r_{i+1},a_i|r_i,a_i),\qquad  a_0=\tfrac{1}{2},
\end{equation}
and
\begin{equation}
\quad b_{i+1}=F(r_{i+1},F(r_{i+1},F(r_{i+1},1|r_i,b_i)|r_i,b_i)|r_i,b_i) ,\qquad  b_0=\tfrac{1}{2}.
\end{equation}
In particular, we have the bounds
\begin{equation}\label{imp_up_bd_init}
K(r) \leq \frac{1}{2}+\frac{r^2}{20}\,\,\,\,\,\, \textrm{for } r\leq 1 \qquad \textrm{and}\qquad K(r) \geq 1-\frac{7}{5r^2} \,\,\,\,\,\, \textrm{for } r\geq \frac{39}{10}.
\end{equation}
\end{proposition}

These quantitative bounds are illustrated in Figure \ref{fig:figure_g}.

\begin{figure}
  \centering
  \includegraphics[width=.5\linewidth]{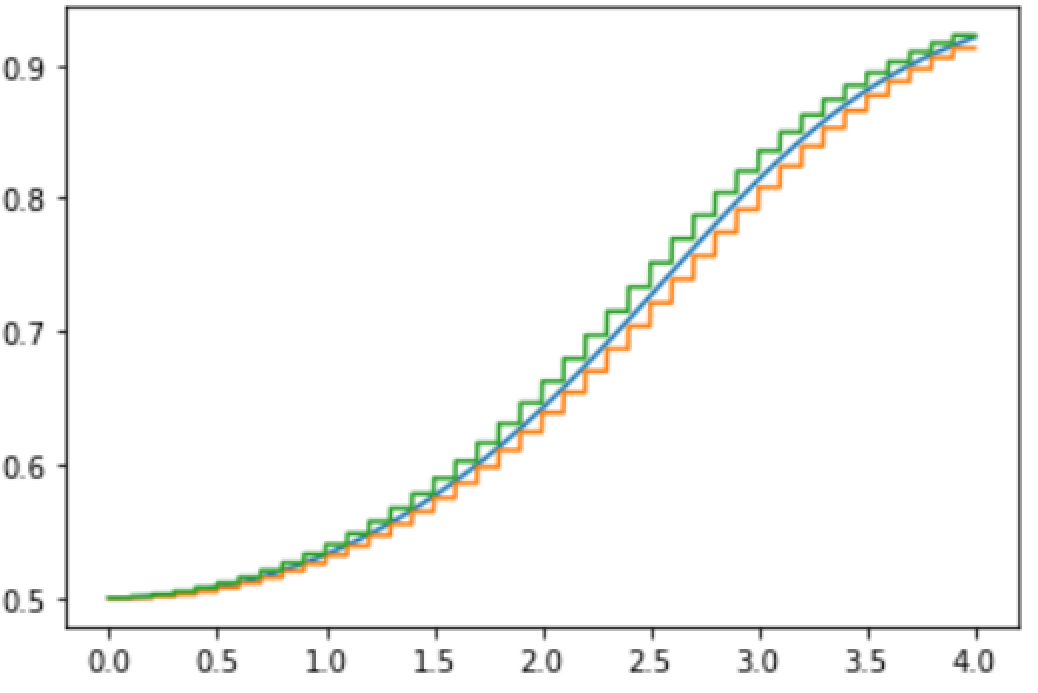}
  \caption{The function $K$ (blue) and the lower and upper bounds (orange and green).}
  \label{fig:figure_g}
\end{figure}

\begin{proof}
To begin with, note that $F_a(r)=F(r,a|\rho,\alpha)$ is the solution of the initial value problem
\begin{equation}
\frac{(rF_a)'}{1+a^2r^2}+F_a=1,\qquad F_a(\rho)=\alpha.
\end{equation}
Now, assuming by induction that \eqref{table_eq} holds for $i$, we will show that it also holds for $i+1$.\\

To prove the lower bound, using the induction hypothesis and the fact that $\psi'\geq 0$ we see that
\begin{equation}
\frac{(rK)'}{1+a_i^2r_i^2}+K\geq 1 \quad(\textrm{for}\,\, r\geq r_i),  \qquad K(r_i)\geq a_i.
\end{equation}
By comparison this implies
\begin{equation}
K(r)\geq F(r,a_i|r_i,a_i) \quad(\textrm{for}\,\, r\geq r_i).
\end{equation}
In particular, this shows that $K(r)\geq a_{i+1}$ for $r\geq r_{i+1}$.\\

To prove the upper bound, note first that since
\begin{equation}
\frac{(rK)'}{1+r^2}+K\leq 1,\qquad K(r_i)\leq b_i,
\end{equation}
by comparison we get
\begin{equation}
K(r)\leq F(r,1|r_i,b_i) \quad(\textrm{for}\,\, r\geq r_i).
\end{equation}
In particular, setting $c_{i+1}=F(r_{i+1},1|r_i,b_i)$ we get $K(r)\leq c_{i+1}$ for $r\leq r_{i+1}$. This implies the sharper differential inequality
\begin{equation}
\frac{(rK)'}{1+c_{i+1}^2 r^2}+K\leq 1 \quad(\textrm{for}\,\, r\leq r_{i+1}),\qquad K(r_i)\leq b_i,
\end{equation}
so using comparison again we infer that
\begin{equation}
K(r)\leq F(r,c_{i+1}|r_i,b_i) \quad(\textrm{for}\,\, r_i\leq r\leq r_{i+1}).
\end{equation}
Repeating this process once more with the improved constant $\tilde{c}_{i+1} := F(r_{i+1},c_{i+1}|r_i,b_i)$ we infer that
\begin{equation}
K(r)\leq F(r,\tilde{c}_{i+1}|r_i,b_i) \quad(\textrm{for}\,\, r_i\leq r\leq r_{i+1}).
\end{equation}
In particular, this shows that $K(r)\leq b_{i+1}$ for all $r\leq r_{i+1}$.\\

Finally, let us explain how \eqref{imp_up_bd_init} follows. Choosing $h=1$ for $r\leq 1$ we get
\begin{equation}
K(r)\leq F(r,e^{-1/2}|0,\tfrac{1}{2})\leq \frac{1}{2}+\frac{r^2}{20}.
\end{equation}
On the other hand, choosing  $h=\tfrac{1}{10}$ we get after $34$ steps (which are done most conveniently with a computer or calculator, but also can be done computing patiently by hand) that
\begin{equation}\label{a34bd}
a_{34} \geq \frac{173}{200}. 
\end{equation}
For $r\geq \tfrac{39}{10}$ this implies
\begin{equation}
K(r)\geq F(r,a_{34}|\tfrac{34}{10},a_{34}) \geq 1-\frac{7}{5r^2},
\end{equation}
and thus concludes the proof of the proposition.
\end{proof}

\bigskip 

As a corollary we obtain important information about the coefficient functions of the model operator $L_0'$ on the bowl soliton, see equation \eqref{model_op_bowl} below. Specifically, these coefficient functions are given by
\begin{equation}
\alpha=\frac{1}{1+r^2K^2},\,\,\, \beta=\frac{1}{r}-\frac{4(1-K)rK}{1+r^2K^2},\,\,\,\gamma=2(1-K)\left(1-\frac{2(1-K)}{1+r^2K^2}\right)-\left(1+\frac{2r^2K^2}{1+r^2K^2}\right)\frac{1}{r^2}.
\end{equation}
Moreover, when constructing a supersolution in \eqref{delta_coeff} we will also encounter their combination
\begin{equation}
\delta=2r^2\alpha+2\left(r-\frac{22}{10}\right)r^2\beta-\left(5-\left(r-\frac{22}{10}\right)^2\right)r^2\gamma.
\end{equation}
The function $\delta(r)$, which starts at $\delta(0)=4/25$, is illustrated in Figure \ref{fig:figure_delta}.

\begin{figure}
  \centering
  \includegraphics[width=.5\linewidth]{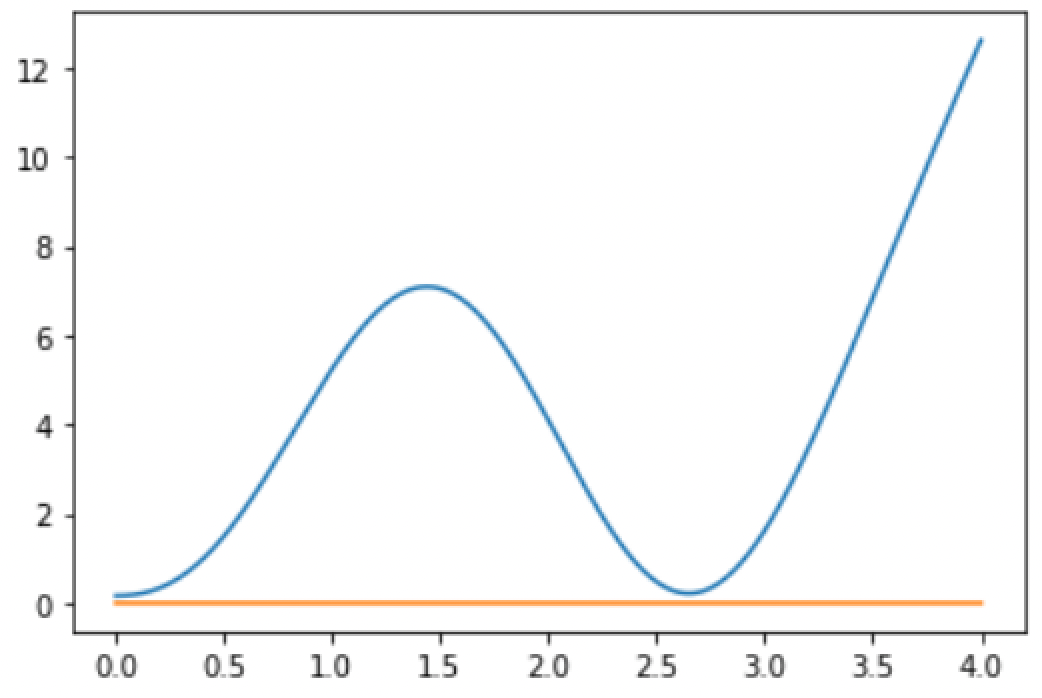}
  \caption{The function $\delta$ (blue) is strictly larger than $0$ (orange).}
  \label{fig:figure_delta}
\end{figure}

\begin{corollary}[coefficient estimate]\label{app_cor}
For $r\geq \tfrac{39}{10}$ we have
\begin{equation}
\alpha(r) \leq \frac{5}{4r^2}, \qquad \beta(r) \geq \frac{5}{9r}, \qquad \gamma(r) \leq -\frac{1}{25r^2},
\end{equation}
and for $r\leq 4$ we have
\begin{equation}\label{L_bd_eq}
\delta(r)\geq\frac{1}{100}.
\end{equation}
\end{corollary}

\begin{proof}
Since by Proposition \ref{app_prop} (quantitative bounds) for $r\geq \tfrac{39}{10}$ we have $K(r)\geq 1-\tfrac{7}{5r^2}\geq \tfrac{9}{10}$, the bounds for $\alpha$ and $\beta$ readily follow.
Moreover, taking also into account the monotonicity of the function $x\mapsto \frac{x}{1+x}$ we see that 
$\frac{r^2K^2}{1+r^2K^2}\geq \frac{369}{400}$ for $r\geq \tfrac{39}{10}$, and so the bound for $\gamma$ follows as well.
Finally, assuming now $r\leq 4$, note that for $r\in[r_i,r_{i+1}]$ if we know that $a_i\leq K(r)\leq b_{i+1}$ then we get
\begin{multline}
\delta(r)\geq\tfrac{2r_i^2}{1+b_{i+1}^2r_{i+1}^2}+2\min\left\{ \left(r_i-\tfrac{22}{10}\right)r_i^2\underline{\beta_i},(r_{i+1}-\tfrac{22}{10})r_{i+1}^2\underline{\beta_i},(r_i-\tfrac{22}{10})r_i^2\overline{\beta_i} ,(r_{i+1}-\tfrac{22}{10})r_{i+1}^2\overline{\beta_i} \right\}\\
+\left(5-\max\left\{\left(r_i-\tfrac{22}{10}\right)^2,\left(r_{i+1}-\tfrac{22}{10}\right)^2\right\}\right)\left(\left(1+\tfrac{2r_i^2a_i^2}{1+r_i^2a_i^2}\right)\tfrac{1}{r_{i+1}^2}-2(1-a_i) + 4\tfrac{(1-b_{i+1})^2}{1+r_{i+1}^2b_{i+1}^2}\right),
\end{multline}
where
\begin{equation}
\underline{\beta_i}=\tfrac{1}{r_{i+1}}-\tfrac{4b_{i+1}(1-a_{i})r_{i+1}}{1+a_{i}^2r_{i}^2},\qquad \overline{\beta_i}= \tfrac{1}{r_{i}}-\tfrac{4a_i(1-b_{i+1})r_i}{1+b_{i+1}^2r_{i+1}^2}.
\end{equation}
Applying Proposition \ref{app_prop} (quantitative bounds) we conclude that $\delta(r)\geq \tfrac{1}{100}$ for $r\leq 4$.
\end{proof}

\bigskip

\section{An enhanced quadratic concavity estimate for the oval-bowls}\label{sec_enh}

In this section, we derive an enhanced quadratic concavity estimate for the renormalized profile function of nontrivial noncollapsed translators $M\subset\mathbb{R}^4$.

\begin{lemma}[third derivative decay]\label{lem:q_zzz/z.cyl} For every $\eps>0$ there exist $\kappa>0$ and $\tau_\ast>-\infty$, such that if $M$ is $\kappa$-quadratic (see \cite[Definition 1.4]{CHH_translators}) at some time $\tau_0\leq\tau_\ast$, then
\begin{equation}
\sup_{|y|\leq |\tau|^{1/2}}\left|\frac{v_{yyy}}{y}\right| \leq \frac{\eps }{|\tau|}\qquad \mathrm{for}\,\, \tau\leq\tau_0.
\end{equation}
\end{lemma}

\begin{proof}Observe first that thanks to \cite[Corollary 1.10]{CHH_profile}, taking also into account \cite[Theorem 1.8]{CHH_profile}, the estimate holds outside the parabolic region. Namely, there exist $L<\infty$, $\kappa>0$ and $\tau_\ast>-\infty$, such that if $M$ is $\kappa$-quadratic at some time $\tau_0\leq\tau_\ast$, then
\begin{equation}
\sup_{L\leq y \leq |\tau|^{1/2}}\left|\frac{v_{yyy}}{y}\right| \leq \frac{\eps }{|\tau|}\qquad \mathrm{for}\,\, \tau\leq\tau_0.
\end{equation}
Also, thanks to the $\mathbb{Z}_2$-symmetry we have $v_{yyy}=0$ at $y=0$, so L'Hopital's rule yields
\begin{equation}
\lim_{y\to 0}\frac{v_{yyy}}{y}= v_{yyyy}|_{y=0},
\end{equation}
which is in particular bounded in absolute value by $\eps/|\tau|$ thanks to the asymptotics in the parabolic region. Finally, if at any $\tau\leq \tau_0$ the function $\varphi=y^{-1}v_{yyy}$ has a critical point at $y_0\in (0,L)$, then
\begin{equation}
\varphi(y_0,\tau)=v_{yyyy}(y_0,\tau),
\end{equation}
which is bounded again in absolute value by $\eps/|\tau|$. This proves the lemma.
\end{proof}

\begin{proposition}[enhanced third derivative estimate]\label{thm:enhanced.third}
For every $\theta>0$ there exist $C<\infty$, $\kappa>0$ and $\tau_\ast>-\infty$, such that if $M$ is $\kappa$-quadratic at some time $\tau_0\leq\tau_\ast$, then
\begin{equation}
\sup_{v(y,\tau)\geq \theta}\frac{|(v^2)_{yyy}|}{|y|(1+|y|)}\leq \frac{C}{|\tau|^2}\qquad \mathrm{for}\,\, \tau\leq\tau_0.
\end{equation}
\end{proposition}

\begin{proof}
Observe first that thanks to \cite[Corollary 1.10]{CHH_profile}, taking also into account \cite[Theorem 1.8]{CHH_profile}, 
the estimate holds for $\theta\leq v(y,\tau)\leq \sqrt{2}-\theta$. More precisely, there exist constants $D<\infty$, $\kappa>0$ and $\tau_\ast>-\infty$, such that $q(z,t)=v(|\tau|^{1/2}z,\tau)^2$, assuming $\kappa$-quadraticity at $\tau_0\leq\tau_\ast$, satisfies
\begin{equation}\label{q_zzz_bdd}
\sup_{q(z,\tau)\geq \theta^2}\left(|q_{zz}|+ |\tau|^{-1/2}|q_{zzz}|\right)\leq D \qquad \mathrm{for}\,\, \tau\leq\tau_0.
\end{equation}
Now, given any $\delta>0$, in the region $\{ 0 < z\leq 1\}$ we consider the function
\begin{equation}
\psi(z,\tau)=\frac{q_{zzz}}{z^{1-\delta}}\pm \eta(|\tau|^{1/2}z),
\end{equation}
where we fix a smooth monotone function $\eta$ with $0\leq \eta''\leq 1/2$ satisfying $\eta(y)= 2$ for $y\in [0,5]$ and $\eta(y)=Dy$ for $y \geq 5D$.
In particular, note that $\lim_{z\to 0} \psi(z,\tau) =\pm 2$ and $\pm \psi(1,\tau) \geq 0$.

To compute the evolution equation for $\psi$, note that similarly as in \cite[Equation (73)]{CHH_profile} we have
\begin{multline}
q_{zz\tau}= \frac{4qq_{zzzz}-4q_zq_{zzz}}{4q|\tau|+q_z^2}-\frac{z}{2} \left(1+\frac{1}{|\tau|}\right)q_{zzz}-\frac{q_{zz}}{|\tau|}
+\frac{4(q_z^2-2qq_{zz})(2|\tau|+q_{zz})}{(4q|\tau|+q_z^2)^2} q_{zz} \\
+\frac{12q_zq_{zzz}(q_z^2-2qq_{zz}) }{(4q|\tau|+q_z^2)^2}-\frac{16q_z^2(q_z^2-2qq_{zz})(2|\tau|+q_{zz})^2}{(4q|\tau|+q_z^2)^3}+E,     
\end{multline}
where the error term $E$ comes from the $e^\tau \mathcal{N}v$-term in the translator profile evolution from \cite[Proposition 5.3]{CHH_translators}, and satisfies the estimate $|E|+z^{-1}|E_z| \leq e^{\tfrac{99}{100}\tau}$. Differentiating the above equation again yields
\begin{equation}
q_{zzz\tau}=\frac{4qq_{zzzzz}}{4q|\tau|+q_z^2}-\frac{z}{2} \left(1+\frac{1}{|\tau|}\right)q_{zzzz}+Aq_{zzzz}-\frac{1}{2}\left(1+\frac{3}{|\tau|}\right)q_{zzz}+Bq_{zzz}+R,
\end{equation}
for certain coefficients $A$, $B$ and $R$, which are explicitly given by
\begin{equation}
A=-\frac{8qq_z(2|\tau|+q_{zz})}{(4q|\tau|+q_z^2)^2}+\frac{12q_z(q_z^2-2qq_{zz}) }{(4q|\tau|+q_z^2)^2},
\end{equation}
and 
\begin{align}
B=&-\frac{4q_{zz}}{4q|\tau|+q_z^2}+\frac{8q_z^2(2|\tau|+q_{zz})}{(4q|\tau|+q_z^2)^2}-\frac{8q(2|\tau|+q_{zz})q_{zz}}{(4q|\tau|+q_z^2)^2}+\frac{8(q_z^2-2qq_{zz})(|\tau|+q_{zz})}{(4q|\tau|+q_z^2)^2}\nonumber \\
&+\frac{12q_{zz}(q_z^2-2qq_{zz}) }{(4q|\tau|+q_z^2)^2}-\frac{24qq_zq_{zzz}}{(4q|\tau|+q_z^2)^2}-\frac{48q_z^2(q_z^2-2qq_{zz})(2|\tau|+q_{zz}) }{(4q|\tau|+q_z^2)^3}\\
&+\frac{32qq_z^2(2|\tau|+q_{zz})^2}{(4q|\tau|+q_z^2)^3}-\frac{32q_z^2(q_z^2-2qq_{zz})(2|\tau|+q_{zz})}{(4q|\tau|+q_z^2)^3},\nonumber
\end{align}
and
\begin{align}
R=-\frac{48q_zq_{zz}(q_z^2-2qq_{zz})(2|\tau|+q_{zz})^2}{(4q|\tau|+q_z^2)^3}+\frac{96q_z^3(q_z^2-2qq_{zz})(2|\tau|+q_{zz})^3}{(4q|\tau|+q_z^2)^4}+E_z.
\end{align}
Using \cite[Theorem 1.8]{CHH_profile} and \eqref{q_zzz_bdd} for $|z|\leq 1$ we can estimate
\begin{align}
&|A| \leq Cz|\tau|^{-1}, && |B|\leq C|\tau|^{-1}, && |R| \leq Cz|\tau|^{-1}.
\end{align}
Thus, for $|z|\leq 1$ the evolution equation for $q_{zzz}$ takes the form
\begin{equation}
q_{zzz\tau} = \frac{4qq_{zzzzz}}{4q|\tau|+q_z^2}-\frac{z}{2} \left(1+O(|\tau|^{-1})\right)q_{zzzz}-\frac{1}{2}\left(1+O(|\tau|^{-1})\right)q_{zzz}+z\,O(|\tau|^{-1}).
\end{equation}
To proceed, note that setting $\varphi=z^{\delta-1}q_{zzz}$ for $z\neq 0$ we have
 \begin{align}
& q_{zzzz}=(1-\delta)z^{-\delta}\varphi+z^{1-\delta}\varphi_z, && q_{zzzzz}=-\delta(1-\delta)z^{-1-\delta} \varphi+ 2(1-\delta)z^{-\delta}\varphi_z+z^{1-\delta}\varphi_{zz}.
 \end{align}
Hence, in $\{0<z \leq 1\}\cap \{ \mp \varphi \geq 0\}$ we obtain
\begin{equation}
\mp \varphi_\tau \leq \mp \frac{4q(\varphi_{zz}+2(1-\delta)z^{-1}\varphi_z) }{4q|\tau|+q_z^2} \pm \frac{1}{2} \left(1+O(|\tau|^{-1})\right)((1-\delta)\varphi+z\varphi_z) \pm \left(\frac{1}{2}-\frac{C}{|\tau|}\right)\varphi+ \frac{C}{|\tau|}.
\end{equation}
Also, note that
\begin{align}
& \psi_\tau=\varphi_\tau\mp \tfrac{1}{2} z|\tau|^{-\frac{1}{2}}\eta', && \psi_z=\varphi_z\pm |\tau|^{\frac{1}{2}}\eta', && \psi_{zz}=\varphi_{zz}\pm |\tau|\eta''.
\end{align}
Hence, if $\mp \psi(\cdot,\tau_1)$ for some $\tau_1\leq \tau_0$ attains a positive local maximum at some $z_1\in (0,1)$, then
\begin{align}
\mp \psi_\tau(z_1,\tau_1)\leq  \eta''+ \tfrac52\tfrac{\eta'}{z_1|\tau_1|^{\frac{1}{2}}} - \left(\tfrac{1}{2} -\tfrac{C}{|\tau_1|}\right)((1-\delta)\eta+z_1|\tau_1|^{\frac{1}{2}}\eta')-\tfrac{1}{2}\left(1-\tfrac{C}{|\tau_1|}\right) \eta + \tfrac{C}{|\tau_1|} \leq -1.
\end{align}
Consequently, setting $\bar{\psi}(\tau)=\sup_{|z|\leq 1} \psi(z,\tau)$, if $\mp \bar{\psi}(\tau_1)> 0$ for some $\tau_1\leq \tau_0$, then we would get $\mp \bar{\psi}(\tau)\geq |\tau-\tau_1|$ for all $\tau\leq \tau_1$, contradicting
Lemma \ref{lem:q_zzz/z.cyl} (third derivative decay). Since $\delta>0$ was arbitrary, this shows that 
\begin{equation}
\frac{|q_{zzz}|}{z} \leq \eta(|\tau|^{1/2}z)\leq C(1+|\tau|^{1/2}z)
\end{equation}
for all $\tau\leq \tau_0$ and $0<z\leq 1$, which concludes the proof.
\end{proof}
 
\begin{corollary}[enhanced quadratic concavity estimate]\label{cor:improved.q_zz}
For every $\eps>0$, there exist $\delta>0$, $\kappa>0$ and $\tau_\ast>-\infty$, such that if $M$ is $\kappa$-quadratic at some time $\tau_0\leq\tau_\ast$, then
\begin{equation}
\sup_{ |y|\leq \delta|\tau|^{1/3}} (v^2)_{yy} \leq -\frac{2-\varepsilon }{|\tau|}\qquad \mathrm{for}\,\, \tau\leq\tau_0.
\end{equation}
\end{corollary} 

\begin{proof}Thanks to the standard asymptotics in the parabolic region we have
\begin{equation}
\sup_{|z|\leq |\tau|^{-1/2}} q_{zz}\leq -2+\frac{\eps}{2}.
\end{equation}
On the other hand, thanks to Proposition \ref{thm:enhanced.third} (enhanced third derivative estimate) we get
\begin{equation}
\sup_{|\tau|^{-1/2}\leq |z|\leq \delta|\tau|^{-1/6}}q_{zz}(z,\tau)\leq q_{zz}(|\tau|^{-1/2},\tau)+C \delta^3.
\end{equation}
Combining these two estimates, the assertion follows.
\end{proof}

\bigskip

\section{The gradient estimate for the linearized translator equation}\label{sec_grad}

To prove the gradient estimate, we will construct a suitable supersolution for the operator
\begin{equation}\label{Ltip_Der}
L_{\tip}'[P]=-P_{\tau}+\frac{P_{vv}}{1+Y_{v}^2}+\left(\frac{1}{v}-\frac{v}{2}-4\frac{Y_{vv}Y_{v}}{(1+Y_{v}^2)^2}\right)P_v-\left(\frac{1}{v^2}+2\left(\frac{Y_{vv}Y_{v}}{(1+Y_{v}^2)^2}\right)_v\right)P.
\end{equation}
Note that since $W$  satisfies the linearized translator equation in tip gauge, we have
\begin{equation}\label{LpWv}
L_{\tip}'[W_v]+e^\tau (\mathcal{F}W)_v=G_v.
\end{equation}
We begin by constructing a supersolution in the collar region:

\begin{proposition}[supersolution in collar region]\label{prop_super_sol_collar} 
If $A<\infty$ is a sufficiently large constant, then for $\tau\leq \tau_0$ in the collar region $v\in [\ell/\sqrt{|\tau|},\theta]$ we have
\begin{equation}
L_{\tip}'\left[A+(v-v^2)|\tau|^{1/2}\right] \leq -|\tau|^{1/2}-\frac{1}{v^2}.
\end{equation} 
\end{proposition}

\begin{proof}
Setting $P=A+(v-v^2)|\tau|^{1/2}$, let us first observe that three terms have the good sign, specifically
\begin{equation}
P_{vv}\leq 0,\qquad -vP_v\leq 0,\qquad -Y_{vv}Y_v P_v\leq 0.
\end{equation}
Also note that
\begin{equation}
-P_{\tau}\leq \frac{1}{|\tau|^{1/2}}.
\end{equation}
Moreover, a direct computation shows that
\begin{equation}
\frac{1}{v}P_v-\frac{1}{v^2}P=-|\tau|^{1/2}-\frac{A}{v^2}.
\end{equation}
Finally, since by the enhanced profile estimates from \cite[Corollary 1.10]{CHH_profile} we have
\begin{equation}
\left|\left(\frac{Y_{vv}Y_v}{(1+Y_{v}^2)^2}\right)_v\right| \leq \frac{C}{v^3|\tau|^{1/2}},
\end{equation}
we can estimate
\begin{equation}
\left|\left(\frac{Y_{vv}Y_v}{(1+Y_{v}^2)^2}\right)_vP\right| \leq\frac{C(1+A/\ell)}{v^2}.
\end{equation}
Combining the above facts, the assertion follows.
\end{proof}

Next, we will construct a suitable supersolution on the bowl soliton for the model operator
\begin{align}\label{model_op_bowl}
L'_0[p]=\frac{p_{rr}}{1+\varphi_r^2}+\left(\frac{1}{r}-4\frac{\varphi_{rr}\varphi_{r}}{(1+\varphi_r^2)^2}\right)p_{r}-\left(\frac{1}{r^2}+2\left(\frac{\varphi_{rr}\varphi_{r}}{(1+\varphi_{r}^2)^2}\right)_{r}\,\,\right)p,
\end{align}
where $\varphi=\varphi(r)$ denotes the profile function of the bowl soliton with tip speed $1$. 

\begin{proposition}[supersolution on bowl]\label{super_sol_bowl}
There exists a smooth positive function $p=p(r)$ such that 
\begin{equation}
L'_0[p] \leq -\frac{1}{r^2}.
\end{equation}
Moreover, we can arrange that $p(r)$ is constant for $r$ sufficiently large.
\end{proposition}

\begin{proof}
Writing our model operator in the form
\begin{equation}
L_0' = \alpha(r)\partial^2_r + \beta(r)\partial_r + \gamma(r),
\end{equation}
by Corollary \ref{app_cor} (coefficient estimate) we have
\begin{equation}\label{delta_coeff}
2\alpha(r)+ 2\beta(r)\left(r-\tfrac{22}{10}\right)-\gamma(r)\left(5-\left(r-\tfrac{22}{10}\right)^2\right)\geq \frac{1}{100r^2}\qquad (\textrm{for } r\leq 4),
\end{equation}
and
\begin{equation}\label{alphabetagamma}
\alpha(r) \leq \frac{5}{4r^2}, \qquad \beta(r) \geq \frac{5}{9r}, \qquad \gamma(r) \leq -\frac{1}{25r^2}\qquad (\textrm{for } r\geq \tfrac{39}{10}).
\end{equation}
Motivated by this, let us define
\begin{equation}
p_1(r) = 5-\left(r-\tfrac{22}{10}\right)^2\qquad \textrm{and} \qquad p_2^{a,b,\delta}(r)=a+b\int_{\frac{39}{10}+\delta}^r e^{\frac29 \left(\left(\frac{39}{10}+\delta\right)^2-\rho^2\right)}\, d\rho\, .
\end{equation}
Then the inequality \eqref{delta_coeff} takes the form
\begin{equation}
L_0'[p_1]\leq  -\frac{1}{100r^2}\qquad (\textrm{for } r\leq 4).
\end{equation}
On the other hand, a direct computation shows that
\begin{equation}
5-\left(\tfrac{39}{10}-\tfrac{22}{10}\right)^2-2\left(\tfrac{39}{10}-\tfrac{22}{10}\right)\int_{\frac{39}{10}}^\infty e^{\frac29 \left(\left(\frac{39}{10}\right)^2-\rho^2\right)}\, d\rho \geq\frac{1}{3}.
\end{equation}
By continuity we thus infer that if
\begin{equation}\label{ab_parameters}
\left|a-p_1(\tfrac{39}{10}+\delta)\right|+\left|b- p_1'(\tfrac{39}{10}+\delta)\right|\leq \bar{\delta}, 
\end{equation}
for $\delta,\bar{\delta}>0$ sufficiently small, then $p_2=p_2^{a,b,\delta}$ satisfies
\begin{equation}
p_2\geq \frac14.
\end{equation}
Moreover, observing also that since $b<0$ we have
\begin{equation}
p_2 ' <0 \qquad \textrm{and} \qquad p_2'' = -\frac49 r p_2' >0,
\end{equation}
we can use \eqref{alphabetagamma} to estimate
\begin{equation}
L'_0[p_2]\leq -\frac{1}{100r^2} \qquad (\textrm{for } r\geq \tfrac{39}{10}+\delta).
\end{equation}
Hence, fixing $\delta,\bar{\delta}>0$ sufficiently small, by the gluing principle from \cite[Remark 3.7]{BuzanoHaslhoferHershkovits} we can find a smooth positive function $\bar{p}$ and parameters $a,b$ satisfying \eqref{ab_parameters}, such that
\begin{equation}
\bar p=p_1\,\,\,\textrm{ for } r\leq \tfrac{39}{10}\qquad \textrm{and} \qquad \bar p=p_2^{a,b,\delta} \,\,\,\textrm{ for } r\geq \tfrac{39}{10}+\delta,
\end{equation}
and
\begin{equation}
L_0'[\bar p]\leq -\frac{1}{200r^2}\quad\textrm{ for all } r> 0.
\end{equation}
Setting 
\begin{equation}
p=400\left(\chi \bar{p}+(1-\chi)c\right),
\end{equation}
where $\chi$ is a suitable cutoff function and $c=\lim_{r\rightarrow \infty} \bar{p}(r)$, this implies the assertion.
\end{proof}

Combining the above two results we obtain a supersolution in the entire tip region:

\begin{corollary}[supersolution in tip region]\label{cor_supersol_entire_tip}
There exists a smooth positive function $S=S(v,\tau)$ such that for all $\tau\leq\tau_0$ and $v\leq \theta$ we have
\begin{equation}\label{Sgbd}
L_{\tip}'[S]\leq -|\tau|^{1/2}-\frac{1}{v^2}\qquad \textrm{and} \qquad C^{-1}(1+v|\tau|^{1/2}) \leq S(v,\tau) \leq C(1+v|\tau|^{1/2}).
\end{equation}
\end{corollary}

\begin{proof} Making the ansatz $P(v,\tau)=2p\left(v\sqrt{|\tau|/2}\right)$, a direct computation shows that
\begin{equation}\label{Lanz}
L_{\tip}'[P]=|\tau|\left(\frac{p_{rr}}{1+Z_r^2}+\left(\frac{1}{r}-\frac{4Z_{rr}Z_{r}}{(1+Z_r^2)^2}\right)p_{r}-\left(\frac{1}{r^2}+\left(\frac{2Z_{rr}Z_{r}}{(1+Z_{r}^2)^2}\right)_{r}\,\right)p\right)-\left(1-\frac{1}{|\tau|}\right)r p_r,
\end{equation}
where $r=v\sqrt{|\tau|/2}$, and where $Z$ denotes the zoomed in profile function defined by
 \begin{equation}\label{def_Z}
Z(r,\tau)= \sqrt{|\tau|/2}\left(Y\big(r\sqrt{2/|\tau|},\tau\big)-Y(0,\tau)\right).
\end{equation}
Hence, letting $p=p(r)$ be the function from Proposition \ref{super_sol_bowl} (supersolution on bowl) and using also the asymptotics from \cite[Theorem 2.1]{CHH_lin_trans} for $\tau\leq \tau_0$ and $v\leq \ell/\sqrt{|\tau|}$ we get
\begin{equation}\label{P_0_ineq}
L_{\tip}'[P] \leq -\frac{1}{v^2}.
\end{equation}
Moreover, using again \cite[Theorem 2.1]{CHH_lin_trans} we see that for $\tau\leq \tau_0$ and $v\leq \ell/\sqrt{|\tau|}$ we have
\begin{equation}\label{vvs_sol}
L_{\tip}'\left[(v-v^2)|\tau|^{1/2}\right] \leq C|\tau|.
\end{equation}
Hence, for any $\tilde{A}<\infty$ sufficiently large,  for $\tau\leq \tau_0$ and $v\leq \ell/\sqrt{|\tau|}$ we obtain
\begin{equation}
L_{\tip}'\left[\tilde{A}P+(v-v^2)|\tau|^{1/2}\right]\leq -\frac{1}{v^2}.
\end{equation}
Recalling also that, possible after increasing $\ell$, the function $p(r)$ is constant for $r\geq\ell/2$, together with Proposition \ref{prop_super_sol_collar} (supersolution in collar region) this implies the assertion.
\end{proof}

Finally, away from the tip region we have:

\begin{proposition}[supersolution in cylindrical region]\label{prop:grad.super.cyl} For all $\tau\leq \tau_0$ and  $\theta/2\leq v\leq \sqrt{2}-\ell^2/(10|\tau|)$ the function $R(v,\tau)=(2-v^2)^{-\frac{1}{2}}|\tau|^{\frac{1}{2}}$ satisfies
\begin{equation}
L_{\tip}'[R] \leq -\tfrac{1}{5}(2-v^2)R.
\end{equation}
\end{proposition}

\begin{proof}Plugging our function $R=(2-v^2)^{-\frac{1}{2}}|\tau|^{\frac{1}{2}}$ into \eqref{Ltip_Der} we get
\begin{equation}
L_{\tip}'[R]=\left(\frac{1}{2|\tau|}+\frac{2(1+v^2)}{(2-v^2)^2(1+Y_v^2)}+\frac{1}{2}-\frac{4vY_vY_{vv}}{(2-v^2)(1+Y_v^2)^2}-\frac{1}{v^2}-\left(\frac{2Y_vY_{vv}}{(1+Y_v^2)^2}\right)_v \right) R.
\end{equation}
Now, by \cite[Theorem 1.8]{CHH_profile} we have $v^2v_y^2\leq (1+\eps)(2-v^2)/|\tau|$, so we can estimate
\begin{equation}
\frac{2(1+v^2)}{(2-v^2)^2(1+Y_v^2)}=\frac{2(1+v^2)v_y^2}{(2-v^2)^2(1+v_y^2)}\leq \frac{2}{2-v^2}\left(1+\frac{1}{v^2}\right)\frac{1+\eps}{|\tau|}.
\end{equation}
Next, using Corollary \ref{cor:improved.q_zz} (enhanced quadratic concavity estimate) we see that
\begin{equation}
-\frac{4vY_vY_{vv}}{(2-v^2)(1+Y_v^2)^2}=\frac{2}{2-v^2}\frac{(v^2)_{yy}-2v_y^2}{(1+v_y^2)^2}\leq -\frac{2}{2-v^2}\frac{2-\eps}{|\tau|}1_{\{y\leq \delta |\tau|^{1/3}\}} + Ce^\tau.
\end{equation}
Finally, using Proposition \ref{thm:enhanced.third} (enhanced third derivative estimate), taking also into account \cite[Theorem 1.8 and Theorem 1.9]{CHH_profile}, we can estimate
\begin{equation}
-\left(\frac{2Y_vY_{vv}}{(1+Y_v^2)^2}\right)_v=\frac{1}{v_y}\left(\frac{2v_{yy}}{(1+v_y^2)^2}\right)_y\leq \frac{Cy}{|\tau|}.
\end{equation}
Combining the above estimates we infer that
\begin{equation}
L_{\tip}'[R]\leq \left(\frac{v^2-2}{2v^2}+\frac{C}{y^2}1_{\{y> \delta |\tau|^{1/3}\}} +\frac{Cy}{|\tau|} \right) R.
\end{equation}
This implies the assertion.
\end{proof}
 
To glue these supersolutions together, we fix a monotone smooth function $\chi:\mathbb{R}\rightarrow \mathbb{R}_{+}$ with $\chi(v)=0$ for $v\leq {\theta}/2$ and $\chi(v)=1$ for $v\geq \theta$, and set
\begin{equation}\label{b_def}
B=\chi R + (1-\chi)\Lambda S,
\end{equation}
where $\Lambda=\Lambda(\theta)<\infty$ is a large constant, such that $R\leq \Lambda S$ in the region $\theta/2\leq v\leq \theta$.

\begin{corollary}[global supersolution]\label{prop:grad.super}
For $\tau\leq\tau_0$ and  $v\leq \sqrt{2}-\ell^2/(10|\tau|)$ we have
\begin{equation}
L_{\tip}'[B]\leq 
  \begin{cases}
    -\tfrac{1}{5}(2-v^2)B, & \text{if } v\geq \theta \\
    -|\tau|^{1/2}-v^{-2} & \text{if } v\leq \theta \,.
  \end{cases}
\end{equation} 
\end{corollary}

\begin{proof}
By Proposition \ref{prop:grad.super.cyl} (supersolution in cylindrical region) and Corollary \ref{cor_supersol_entire_tip} (supersolution in tip region) our only remaining task is to estimate the error term $E$ in the formula
\begin{equation}
L_{\tip}'[B]\leq \chi L_{\tip}'[R]+(1-\chi)\Lambda L_{\tip}'[S]+E.
\end{equation}
Explicitly, this error term is given by
\begin{equation}
E= \left[\chi''\frac{1}{1+Y_v^2}-\chi' \frac{4Y_vY_{vv}}{(1+Y_v^2)^2}\right](R-\Lambda S)+2\chi' \frac{1}{1+Y_v^2}(R-\Lambda S)_v,
\end{equation}
so using the derivative estimates from \cite{CHH_profile}
we see that $|E|\leq C$.
This implies the assertion.
\end{proof}

Using this global supersolution, we can now prove a technical version of our main estimate:

\begin{theorem}[gradient estimate - technical version]\label{thm_grad_est}
Let $u$ be a solution of the Dirichlet problem \eqref{bdval_prob} with inhomogeneity $f$, and denote by $W$ and $G$ the associated variation and inhomogeneity in tip gauge. Suppose that $A<\infty$ and $\tau_1\in [-\log (h)^{1/2}+1,-\log(h)^{1/2}+2]$ are such that
\begin{multline}\label{grad_a1}
|\tau|^{\mu-1}\!\!\!\!\!\!\sup_{\tau\in [-\log(h)^{1/2}+2,\tau_0]}  |W_v(v(\ell,\tau),\tau)|+ |\tau|^{\mu-2}\!\!\!\!\!\!\!\!\!\!\!\!\sup_{\tau\in [-\log (h)^{1/2}+1,-\log(h)^{1/2}+2]}  |W_v(v(\ell,\tau),\tau)|\\
+|\tau_1|^{\mu-3/2} \!\!\!\!\!\!\sup_{\theta\leq v\leq v(\ell,\tau_1)}\,\, (\sqrt{2}- v)^{1/2} |W_v(v,\tau_1)|
+|\tau_1|^{\mu-1}\sup_{v\leq \theta}\,\,(1+v|\tau_1|^{1/2})^{-1}|W_v(v,\tau_1)|
\leq A,
\end{multline}
and suppose that for all $\tau\in [-\log (h)^{1/2}+1,\tau_0]$ we have
\begin{equation}\label{grad_a2}
|\tau|^{\mu-1/2}\!\!\!\!\!\! \sup_{\theta\leq v\leq v(\ell,\tau)}\,\, (\sqrt{2}- v)^{-1/2} |G_v-e^\tau (\mathcal{F}W)_v|+ |\tau|^{\mu} \sup_{v\leq \theta}\,\, v^2(1+v^2|\tau|^{1/2})^{-1}  |G_v-e^\tau (\mathcal{F}W)_v| \leq A.
\end{equation}
Then, for all $\tau\in [-\log(h)^{1/4},\tau_0]$ we get
\begin{equation}
|\tau|^{\mu-1/2}\!\!\! \sup_{\theta\leq v\leq v(\ell,\tau)}\,\, (\sqrt{2}- v)^{1/2} |W_v|
+|\tau|^{\mu}\sup_{v\leq \theta}\,\,(1+v|\tau|^{1/2})^{-1}|W_v| \leq CA.
\end{equation}
\end{theorem}

\begin{proof}Fix a convex function $\kappa:\mathbb{R}\to\mathbb{R}_+$ satisfying $|\kappa'|\leq 2\kappa/|\tau|$, such that $\kappa(\tau)=\tfrac{|\tau|^2}{(\log h)^{1/2}}$ for $\tau\leq -\log(h)^{1/2}+2$ and $\kappa(\tau)=2$ for $\tau\geq -(\log h)^{1/4}$. Then, as a consequence of Corollary \ref{prop:grad.super} (global supersolution) we have
\begin{equation}\label{eq_trunc_tip_barr}
L_{\tip}'[\kappa|\tau|^{-\mu} B]\leq \kappa|\tau|^{-\mu} L_{\tip}'[ B] +\tfrac{2+\mu}{|\tau|}\kappa|\tau|^{-\mu} B\leq
  \begin{cases}
    -\tfrac{\kappa}{8}(2-v^2)|\tau|^{-\mu}B, & \text{for }  v\in [\theta,v(\ell,\tau)] \\
    -\tfrac{\kappa}{4}(|\tau|^{1/2}+v^{-2})|\tau|^{-\mu} & \text{for } v\leq \theta \,.
  \end{cases}
  \end{equation}
Now, fixing a sufficiently large numerical factor $\lambda<\infty$, we consider $f_\pm := \lambda A \kappa|\tau|^{-\mu}B \pm W_v$
as a function of $(v,\tau)$, where $v\in [0,v(\ell,\tau)]$ and $\tau\in [\tau_1, \tau_0]$.
Thanks to assumption \eqref{grad_a1} we have $f_\pm(v(\ell,\cdot),\cdot)>0$, and $f_\pm(\cdot,\tau_1)>0$.
Moreover, using equation \eqref{LpWv}, assumption \eqref{grad_a2} and the estimate \eqref{eq_trunc_tip_barr}, we see that $L_{\tip}'[f_\pm]<0$.
Hence, we conclude that $f_\pm > 0$ for all $(v,\tau)$ in the domain under consideration, since otherwise at the first $\tau>\tau_1$ where this failed we would obtain a contradiction with the second derivative test. This proves the theorem.
\end{proof}

Finally, let us explain how this implies Theorem \ref{gradient_intro}, which we restate here for convenience.

\begin{theorem}[gradient estimate - simplified version]
In addition to \eqref{lower_bd_basic_sup_intro} and \eqref{g_growth_basic_sup_intro}, suppose that for all $\tau\in [-\log (h)^{1/2}+1,\tau_0]$ we have
\begin{equation}\label{grad_a2_intro}
|\tau|^{1/2+\mu}\!\!\!\sup_{y \in \left[\ell,Y(\theta,\tau)\right]}\,\, (\sqrt{2}- v)^{-3/2} |g_y(y,\tau)|+ |\tau|^{\mu}\sup_{v\leq \theta}\,\,  v^2(1+v^2|\tau|^{1/2})^{-1}  |G_v(v,\tau)| \leq A.
\end{equation}
Then, for all $\tau\in [-\log(h)^{1/4},\tau_0]$ we get
\begin{equation}
|\tau|^{1/2+\mu}\!\!\! \sup_{y \in \left[\ell,Y(\theta,\tau)\right]}\,\, (\sqrt{2}- v)^{-1/2} |w_y(y,\tau)|
+|\tau|^{\mu}\sup_{v\leq \theta}\,\,(1+v|\tau|^{1/2})^{-1}|W_v(v,\tau)| \leq CA.
\end{equation}
\end{theorem}

\begin{proof}To begin with, recall that by \cite[Corollary 3.4]{CHH_lin_trans} we have the transformation rules
\begin{equation}
W(v,\tau)=-Y_v(v,\tau)w(Y(v,\tau),\tau),\qquad G(v,\tau)=-Y_v(v,\tau)g(Y(v,\tau),\tau),
\end{equation}
and consequently
\begin{align}\label{trans_der1}
W_v = -Y_v^{2} (w_y- v_y^{-1}v_{yy}w),\qquad
G_v = -Y_v^{2} (g_y- v_y^{-1}v_{yy}g),
\end{align}
which can be used to transform bounds from cylindrical gauge to tip gauge, and vice versa.\\

Now, note that by Theorem \ref{inner_outer_intro} (inner-outer estimate), for all $\tau\in [-\log (h)^{1/2},\tau_0]$ we have
\begin{equation}
|\tau|^{\mu}\sup_{v(y,\tau)\geq \theta}  \big(\sqrt{2}+10|\tau|^{-1}-v(y,\tau)\big)^{-1}|w(y,\tau)| +
 |\tau|^{\mu-1/2}\sup_{v\leq \theta} |W(v,\tau)| \leq CA.
\end{equation}
Hence, by the interior estimates from \cite[Proposition 6.3]{CHH_lin_trans}, for $\tau\in [ -\log(h)^{1/2}+1,\tau_0]$ we get
\begin{equation}
|\tau|^{1+\mu}\sup_{|y|\leq \ell}  | w_y(y,\tau) |\leq CA,
\end{equation}
and for $\tau_1\in [ -\log(h)^{1/2}+1,-\log(h)^{1/2}+2]$ we get
\begin{equation}
|\tau_1|^{\mu}\sup_{v(y,\tau_1)\geq \theta} \big(\sqrt{2}+10|\tau_1|^{-1}-v(y,\tau_1)\big)^{-1}|w_y(y,\tau_1)|\leq CA,
\end{equation}
and
\begin{equation}
{|\tau_1|^{\mu-1}}\sup_{\ell/|\tau_1|^{1/2}\leq v\leq \theta}|W_v(v,\tau_1)|\leq CA.
\end{equation}
Moreover, by the interior estimates from \cite[Proposition 6.6]{CHH_lin_trans} we get
\begin{equation}
{|\tau_1|^{\mu-1}}\sup_{ v\leq \ell/|\tau_1|^{1/2}}|W_v(v,\tau_1)|\leq CA.
\end{equation}
Hence, taking also into account that the exponential error term $e^\tau (\mathcal{F}W)_v$ can be controlled using \cite[Corollary 6.5 and Corollary 6.8]{CHH_lin_trans}, we see that the assumptions of Theorem \ref{thm_grad_est} (gradient estimate - technical version) are verified, and thus for all  $\tau\in [-\log(h)^{1/4},\tau_0]$ we get
\begin{equation}
|\tau|^{\mu-1/2} \sup_{\theta\leq v\leq v(\ell,\tau)}\,\, (\sqrt{2}- v)^{1/2} |W_v|
+|\tau|^{\mu}\sup_{v\leq \theta}\,\,(1+v|\tau|^{1/2})^{-1}|W_v| \leq CA.
\end{equation}
Remembering the transformation rules, this implies the assertion.
\end{proof}

\bigskip

\bibliography{LTE}
\bibliographystyle{alpha}

\vspace{5mm}

{\sc Kyeongsu Choi, School of Mathematics, Korea Institute for Advanced Study, 85 Hoegiro, Dongdaemun-gu, Seoul, 02455, South Korea}\\

{\sc Robert Haslhofer, Department of Mathematics, University of Toronto,  40 St George Street, Toronto, ON M5S 2E4, Canada}\\

{\sc Or Hershkovits, Department of Mathematics, University of Maryland, 4176 Campus Dr, College Park, MD 20742, USA and Institute of Mathematics, Hebrew University of Jerusalem, Jerusalem, 91904, Israel}\\

\end{document}